\documentclass[10pt, reqno]{amsart}
\usepackage{amsaddr}
\usepackage{mathtools}
\usepackage{amsmath}
\usepackage{amssymb}

\newtheorem{thm}{Theorem}[section]
\newtheorem{lemma}[thm]{Lemma}
\newtheorem{prop}[thm]{Proposition}
\theoremstyle{definition}
\newtheorem{defn}[thm]{Definition}
\theoremstyle{remark}
\newtheorem{rem}[thm]{Remark}
\numberwithin{equation}{section}

\newcommand{\tr}{\mathrm{Tr}}
\newcommand{\wf}{\mathrm{WF}}
\newcommand{\ilf}{\mathrm{IF}}
\newcommand{\wo}{\mathrm{WO}}
\newcommand{\lht}{\mathrm{LT}}

\DeclareMathOperator{\proj}{proj}
\DeclareMathOperator{\type}{type}
\DeclareMathOperator{\htt}{ht} 
\DeclareMathOperator{\lh}{lh}
\DeclareMathOperator{\Mod}{Mod}
\DeclareMathOperator{\Red}{Red}
\DeclareMathOperator{\Div}{Div}
\DeclareMathOperator{\Tree}{Tree}
\DeclareMathOperator{\Ulm}{Ulm}

\usepackage{hyperref}

\begin{document}

\title[Idealistic equivalence relations and the Lusin derivative]{Idealistic equivalence relations and\\ the Lusin derivative}

\author{Bilge Köksal\textsuperscript{1}, Mateusz Lichman\textsuperscript{2}, and Sławomir Solecki\textsuperscript{1}}
\address{\textsuperscript{1}Department of Mathematics, Malott Hall, Cornell University\\ Ithaca, NY 14853, USA,\\
\textsuperscript{2}Institute of Mathematics, Lodz University of Technology, \\ al. Politechniki 8, 93-590 Łódź, Poland}
\email{bk463@cornell.edu}
\email{lichmanmateusz@gmail.com} 
\email{ss3777@cornell.edu}

\thanks{This paper has been completed while the second author was a Doctoral Candidate in the Interdisciplinary Doctoral School at the Lodz University of Technology, Poland. The research of the second author was funded in whole by National Science Centre, Poland, grant 2025/57/N/ST1/03025.} 

\thanks{The third author was supported by NSF grant DMS-2553992.}



\begin{abstract}
We give a proof, that does not rely on analytic determinacy, of the existence of an idealistic analytic equivalence relation with Borel classes that is not classwise Borel isomorphic to an orbit equivalence relation. 
This result answers a question of Becker, and Calderoni and Motto Ros. Our proof proceeds through an analysis of Borel complexity of certain sets of countable
rooted trees defined via the notion of the Lusin derivative. 
\end{abstract}

\maketitle

\section{Introduction}

The main objects of study in this paper are equivalence relations on standard Borel spaces. We start with recalling two classes of equivalence relations and two methods of comparing equivalence relations.

\begin{defn}
    We say that an equivalence relation $E$ on a standard Borel space $X$ is an \textbf{orbit equivalence relation} if there exists a Polish group $H$ and a Borel action $a:H\times X\to X$ inducing $E$, that is, 
    \[
    xEy \iff \exists h\in H \ (a(h, x)=y).
    \]
\end{defn}

We note that each orbit equivalence relation is analytic and has Borel equivalence classes. 

The following notion of idealistic equivalence relation was introduced in \cite{Kechris_1992} and \cite{Kechris_Louveau_1997}.

\begin{defn}\label{def-idealistic}
    We say that an analytic equivalence relation $E$ on a standard Borel space $X$ is \textbf{idealistic} if there is a map assigning to each equivalence class $C\in X/E$ a ccc $\sigma$-ideal $I_C$ on $C$ such that $C\notin I_C$ and the map $C\mapsto I_C$ is Borel in the following sense: for each Borel $A\subseteq X^2$, the set $A_I\subseteq X$ defined by
    \[
    x\in A_I \iff \{y\in [x]_E: (x, y)\in A\}\in I_{[x]_E}
    \]
    is Borel. 
\end{defn}

Observe that each orbit equivalence relation $E$ is idealistic. Indeed, if $E$ is induced by a Borel action $a$ of a Polish group $H$ on a standard Borel space $X$, then the following assignment $C\mapsto I_C$ witnesses that $E$ is idealistic: for $A\subseteq C$ we let
\[
A\in I_C \iff \{h\in H: a(h, x)\in A\} \text{ is meager in } H,
\]
with $x\in C$ fixed \cite{Kechris_1992}.

The following definition describes the classical notion of Borel reducibility. 

\begin{defn}
    Let $E, F$ be equivalence relations on standard Borel spaces $X, Y$, respectively. We say that $E$ is \textbf{Borel reducible to} $F$, in symbols $E\leq_B F$, if there is a Borel function $f:X\to Y$ such that for all $x, y\in X$
    \[
    xEy \iff f(x)F f(y).
    \]
    In this case we say that $f$ is a \textbf{reduction of $E$ to $F$}. If $E\leq_B F$ and $F\leq_B E$, then we say $E, F$ are \textbf{bireducible}.
\end{defn}

Finally, we recall a notion of comparison of equivalence relations that is stronger than the classical Borel reduction above.

\begin{defn}\label{Definition classwise Borel isomorphic}
    Let $E, F$ be analytic equivalence relations on standard Borel spaces $X, Y$, respectively. We say that $E$ is \textbf{classwise Borel isomorphic to} $F$ if there are Borel reductions $f:X\to Y$, $g:Y\to X$ of $E$ to $F$ and $F$ to $E$, respectively, such that their factorings to the quotient spaces $\hat{f}:X/E\to Y/F$ and $\hat{g}:Y/F\to X/E$ are bijections and $\hat{f}^{-1}=\hat{g}$. In this case we say that $g$ is a \textbf{classwise Borel inverse of} $f$.
\end{defn}

In \cite{BeckersNotes}, Becker constructed an equivalence relation, later called $E_{\mathbb B}$ \cite{calderoni2025structuralresultsidealisticequivalence}, that closely resembles an orbit equivalence relation in that it is analytic, has Borel equivalence classes, and is idealistic. However, assuming analytic determinacy, he proved that $E_\mathbb B$ is not an orbit equivalence relation. Calderoni and Motto Ros \cite{calderoni2025structuralresultsidealisticequivalence}, also assuming analytic determinacy, improved Becker's result by showing that $E_{\mathbb B}$ is not even classwise Borel isomorphic to an orbit equivalence relation. 
Becker \cite[Question 1]{BeckersNotes} and Calderoni and Motto Ros \cite[Question 4.1]{calderoni2025structuralresultsidealisticequivalence}
asked if the use of analytic determinacy can be removed. 

The main result of the current paper removes the assumption of analytic determinacy from the above theorems, thereby answering the questions from 
\cite{BeckersNotes} and \cite{calderoni2025structuralresultsidealisticequivalence}. Furthermore, we define a related, yet easier to describe and, perhaps, more natural equivalence relation $F_{\mathbb B}$, and show, again without assuming analytic determinacy, that it has the same properties as $E_{\mathbb B}$. 
The equivalence relation 
\[
F_{\mathbb B}
\] 
is defined on the space of (codes of) countable abelian $p$-groups, for a fixed prime number $p$. It makes two such groups $H_1, H_2$ equivalent precisely when the quotients of $H_1$ and $H_2$ by their maximal divisible subgroups are isomorphic. 

\begin{thm}\label{thm-Beckers-relation-intr}
    Relations $E_{\mathbb B}$ and $F_{\mathbb B}$ are not classwise Borel isomorphic to an orbit equivalence relation.  
\end{thm}

The slightly stronger version of this result is stated in Theorem~\ref{thm-Beckers-relation}. 

Our proof of this theorem proceeds through an analysis of Borel complexity of certain sets of countable rooted trees. By $\omega$ we denote the set of all natural numbers including $0$. We view countable rooted trees as consisting of finite sequences of elements of $\omega$ with the extension order and with the empty sequence $\emptyset$ as the root, so they are subtrees of $(\omega^{<\omega}, \subseteq)$, that is, subsets of $\omega^{<\omega}$ closed downwards with respect to $\subseteq$ and containing $\emptyset$. By $\tr$ we denote the space of all such trees equipped with the Polish topology inherited from the product space $2^{\omega^{<\omega}}$. For $T\in \tr$, we define the usual derivative 
\begin{equation}\label{derivative_definition}
    D(T)=\{s\in \omega^{<\omega}:\exists n\in \omega \ (s^\smallfrown (n)\in T)\}\cup \{\emptyset\}. 
\end{equation}
For an ordinal $\alpha$, denote by $D^\alpha$ the $\alpha$-th iterate of $D$, where for a limit ordinal $\alpha$, the tree $D^\alpha(T)$ is the intersection of the previously produced trees $D^\xi(T)$ for $\xi<\alpha$. 
The \textbf{Lusin height} of $T$, denoted by $|T|$, is the smallest ordinal $\alpha$  such that 
\[
D^{\alpha+1}(T)= D^\alpha(T). 
\]
For a countable ordinal $\alpha$, we consider the set 
\[ 
\lht_{<\alpha} = \{ T\in \tr: |T|<\alpha\}. 
\]

\begin{thm}\label{T:compl}
    For each countable ordinal $\alpha$ and each $k\in\omega$, the set $\lht_{<\omega\alpha+k+1}$ is $\mathbf{\Pi}^0_{2\alpha+2}$-complete.
\end{thm} 

An expanded version of this result can be found in Theorem~\ref{thm-complete}. 

A statement in \cite[Theorem 6.11]{Zafrany1989BorelIV} implies that $\lht_{<\omega\alpha+k+1}$ is $\mathbf{\Pi}^0_{2\alpha+1}$ for $\alpha\geq \omega$. This statement is incorrect and would have 
made our application impossible---having the index $2\alpha+2$ in Theorem~\ref{T:compl} makes our application in the proof of Theorem~\ref{thm-Beckers-relation-intr} work, while having the index $2\alpha+1$ would not have sufficed. 

Finally, as a consequence of Theorem~\ref{T:compl}, in Theorem~\ref{thm-derivative-operator-complexity} we determine the Borel class of the $\alpha$-th iterate $D^\alpha$ of the Lusin derivative $D$ viewed as a Borel function. Through this result we make a connection with the work of Cenzer and Mauldin \cite{Cenzer1}, \cite{Cenzer2}.

\subsection{Organization of the paper}

In Section~\ref{Section Borel Class of the Lusin Derivative}, we calculate precise Borel complexities of sets of trees of certain Lusin heights. The main results of this section are Theorems~\ref{thm-complete} and \ref{thm-derivative-operator-complexity}. The main difficulty is proving the lower bounds on complexity which is done in Subsection~\ref{subsection-lower-bd}. We introduce some operations on trees called sums, products, and joins, in order to construct the desired Wadge reductions (Lemma~\ref{lemma-lower-bound}). 

In Section~\ref{Section Application to idealistic equivalence relations}, we provide some background on the theory of countable abelian $p$-groups. We define a new version of Becker's equivalence relation (which is equivalent to the original version for our purposes, see Proposition~\ref{prop-alt-becker}) and we prove that it is analytic, idealistic, and has Borel classes, yet it is not classwise Borel isomorphic to an orbit equivalence relation. Two of the three main lemmas (Lemma~\ref{lemma-Becker-Kechris} and Lemma~\ref{lemma_E0_does_not_reduce}) are proved using techniques similar to the ones used in \cite{BeckersNotes} and \cite{calderoni2025structuralresultsidealisticequivalence}. The third main lemma is proved in Subsection~\ref{subsection-tr-gp} (Lemma~\ref{becker-lem}) with methods based on the results from Section~\ref{Section Borel Class of the Lusin Derivative}. This is the point in which we avoid using analytic determinacy. In order to transfer the results of Section~\ref{Section Borel Class of the Lusin Derivative} to the setting of abelian $p$-groups, we study 
an assignment of abelian $p$-groups to trees in Subsection~\ref{subsection-tr-gp}.

\section{Borel Class of the Lusin Derivative}\label{Section Borel Class of the Lusin Derivative}

\begin{defn}\label{defn-ld}
    Let $T$ be a tree on $\omega$. Recall from \eqref{derivative_definition} the definition  of the Lusin derivative of $T$. For every ordinal $\alpha$, $D^\alpha$ is defined by induction on $\alpha$. We let 
    \begin{align*}
        D^0(T)&=T,\\
        D^{\alpha+1}(T)&=D(D^\alpha(T)),\\
        D^\alpha(T)&=\bigcap_{\xi<\alpha}D^\xi(T) \text{ if $\alpha$ is a limit ordinal.} 
    \end{align*}
    The \textbf{Lusin height of $T$} is the least $\alpha<\omega_1$ for which $D^{\alpha+1}(T)=D^\alpha(T)$. We denote it by $|T|$.
\end{defn}

\begin{rem}\label{remark_Lusin_derivative}
    Proposition~\ref{prop-ht} shows that if $T$ is well-founded, then $|T|$ coincides with the length of a tree defined in \cite[Section 2D]{Moschovakis}. We note that the definition of the Lusin derivative of a tree given in \eqref{derivative_definition} differs in an inessential way from the definition of the Lusin derivative used by Zafrany in \cite{Zafrany1989BorelIV} (where $\{\emptyset\}$ is omitted from the formula for $D(T)$). We use formula \eqref{derivative_definition} because we insist that a tree contain a root and that the derivative of a tree be itself a tree. Moreover, the notion of the Lusin height of a tree induced by \eqref{derivative_definition} aligns well with the notion of the Ulm length of a countable abelian $p$-group. This correspondence is central to the results in Section~\ref{Section Application to idealistic equivalence relations}. However, the main results of Section~\ref{Section Borel Class of the Lusin Derivative} (Theorems~\ref{thm-complete} and \ref{thm-derivative-operator-complexity}) remain true for Zafrany's definition of the Lusin derivative (see Remark~\ref{rem-same}).
\end{rem}

For a tree $T\in \tr$ and $s\in T$, we define a tree
\[
T_s = \{t\in \omega^{<\omega}: s^\smallfrown t \in T\}.
\]

\begin{defn}
    Given $T\in \tr$ we define the rank function $\rho_T$ by
    \[
    \rho_T(s)=\sup\{\rho_T(t)+1:t\in T \land t\supsetneq s\}
    \]
    for $s\in T$ such that $s\neq \emptyset$ and $T_s$ is well-founded, and we set $\rho_T(s)=\infty$ for all other $s\in T$. We set $\alpha<\infty$ for any ordinal $\alpha$.
\end{defn}

The following is a standard fact. We supply a proof here just to certify that it also holds for our definition \eqref{derivative_definition}, see Remark~\ref{remark_Lusin_derivative}.
\begin{prop}\label{prop-ht} 
    For $T\in \tr$, we have
    $|T|=\sup \{\rho_T(s)+1:s\in T \land \rho_T(s)\neq \infty\}$.
\end{prop}

\begin{proof}
    By a straightforward induction on $\alpha$ for all $s\in T$ we have 
    \begin{equation}\label{ht}
        s\in D^\alpha(T)\iff \rho_T(s)\geq \alpha.
    \end{equation}
    Set $\beta=\sup \{\rho_T(s)+1:s\in T \land \rho_T(s)\neq \infty\}$. Then we have  
    \[
    \forall s\in T\ ( \rho_T(s)< \beta\lor \rho_T(s)=\infty)
    \] 
    and by \eqref{ht} 
    \[
    \forall s\in D^{\beta}(T)\ (\rho_T(s)=\infty),
    \] 
    which then implies that for all $s\in D^{\beta}(T)$, either $s=\emptyset$ or $s$ extends to an infinite branch of $T$, hence $s\in D^{\beta+1}(T)$. The inclusion $D^{\beta+1}(T)\subseteq D^{\beta}(T)$ follows from the definition of $D$, so we have $|T|\leq \beta$. 

    Now take $\alpha<\beta$. Then there is $s\in T$ such that $\alpha < \rho_T(s)+1 \leq \beta$. Then by \eqref{ht} $s\in D^{\rho_T(s)}(T)\setminus D^{\rho_T(s)+1}(T)$, so $D^{\rho_T(s)}(T)\neq D^{\rho_T(s)+1}(T)$, which implies $D^{\alpha}(T)\neq D^{\alpha+1}(T)$, and hence $|T|\neq \alpha$. Thus $|T|=\beta$, as desired.
\end{proof}
 
By $\wf$ we denote the set of well-founded trees and by $\ilf$ we denote the set of ill-founded trees on $\omega$. For $\alpha<\omega_1$, by 
\[
\lht_{=\alpha}, \ilf_{=\alpha}, \wf_{=\alpha}
\]
we denote the sets of trees, ill-founded trees and well-founded trees on $\omega$ of Lusin height $\alpha$. Analogously we define $\lht_{<\alpha}, \ilf_{<\alpha}, \wf_{<\alpha}$, $\lht_{\leq\alpha}, \ilf_{\leq\alpha}, \wf_{\leq\alpha}$. We calculate the exact Borel complexities of these sets for various values of $\alpha$ using the notion of Wadge reducibility. 

Let $A$, $B$ be subsets of Polish spaces $X$ and $Y$ respectively and $1\leq \xi<\omega_1$. We write $A\in \mathbf{\Pi}^0_\xi(X)$ to indicate that $A$ is a $\mathbf{\Pi}^0_\xi$ subset of $X$. We often do not mention the space $X$ if it is clear from context.  We say that $A$ is \textbf{Wadge reducible to} $B$ if there is a continuous map $f:X\to Y$ such that $f^{-1}(B)=A$. If $B$ is $\mathbf{\Pi}^0_\xi$ and $A$ is Wadge reducible to $B$, then $A$ is also $\mathbf{\Pi}^0_\xi$. A set $B$ is said to be $\mathbf{\Pi}^0_\xi$-hard if $A$ is Wadge reducible to $B$ for any $A\in \mathbf{\Pi}^0_\xi(X)$ where $X$ is a zero-dimensional Polish space. We say that $B$ is $\mathbf{\Pi}^0_\xi$-complete if $B$ is in  $\mathbf{\Pi}^0_\xi$ and is $\mathbf{\Pi}^0_\xi$-hard. By a theorem of Wadge, if $X$ is a zero-dimensional Polish space, a set $A\in \mathbf{\Pi}^0_\xi(X)$ is $\mathbf{\Pi}^0_\xi$-complete if and only if $A\in \mathbf{\Pi}^0_\xi\setminus \mathbf{\Sigma}^0_\xi$ \cite[Theorem 22.10]{kechris1995classical}. Additionally, all of the above holds true when $\mathbf{\Pi}^0_\xi$ is interchanged with $\mathbf{\Sigma}^0_\xi$.

\begin{thm}\label{thm-complete} 
    For $\alpha<\omega_1$ and $k\in \omega$,
    \begin{itemize}
        \item[(i)] $\lht_{<\omega\alpha+k+1}$ is $\mathbf{\Pi}_{2\alpha+2}^0$-complete;
        \item[(ii)]$\lht_{=\omega\alpha}$ is $\mathbf{\Pi}_{2\alpha+2}^0$-complete;
        \item[(iii)] $\lht_{=\omega\alpha+k+1}$ is $\mathbf \Delta_{2\alpha+3}^0$, but it is neither $\mathbf{\Sigma}_{2\alpha+2}^0$ nor $\mathbf{\Pi}_{2\alpha+2}^0$. In fact, it is a difference of two $\mathbf{\Pi}^0_{2\alpha+2}$ sets.
    \end{itemize}
\end{thm}

Let $X, Y$ be Polish spaces and $\alpha<\omega_1$. We say that a function $f:X\to Y$ is \textbf{$\mathbf{\Sigma}^0_\alpha$-measurable} if $f^{-1}(U)\in \mathbf{\Sigma}^0_\alpha(X)$ for every open $U\subseteq Y$.  

\begin{thm}\label{thm-derivative-operator-complexity}
    For $\alpha<\omega_1$ and $k\in \omega$, $D^{\omega\alpha+k}$ is $\mathbf{\Sigma}^0_{2\alpha+1}$-measurable and not $\mathbf{\Sigma}^0_{2\alpha}$-measurable. 
\end{thm}

\begin{rem}
    Theorem~\ref{thm-derivative-operator-complexity} is analogous to the main result of \cite{Cenzer2}, where a notion similar to the Lusin derivative is introduced. Authors of \cite{Cenzer2} define the derived set operator on the space of closed subsets of an uncountable, compact metric space endowed with the Vietoris topology and prove that the $\alpha$-th iterate of this operator is $\mathbf{\Sigma}^0_{2\alpha+1}$-measurable and not $\mathbf{\Sigma}^0_{2\alpha}$-measurable. We note that in this remark as well as in the whole paper we adopt the standard notation of $\mathbf{\Sigma}^0_\alpha$ and $\mathbf{\Pi}^0_\alpha$ sets (used e.g. in \cite{kechris1995classical}), whereas in \cite{Cenzer2} a different definition was employed. 
\end{rem}

\subsection{Upper Bounds on Complexity}\label{subsection-upper-bd}

In this subsection, we calculate the upper bounds on the complexity of the sets in Theorem~\ref{thm-complete}. We start with an auxiliary lemma.

\begin{lemma}\label{D-upper-bd}
    For any $s\in \omega^{<\omega}$, $k\in \omega$, and $0<\alpha<\omega_1$,
    \begin{itemize}
        \item[(i)]$\{T\in \tr :s\in D^{\omega\alpha}(T)\}$ is $\mathbf{\Pi}^0_{2\alpha}$,
        \item[(ii)] $\{T\in \tr: s\in D^{\omega \alpha+k}(T)\}$ is $ \mathbf{\Sigma}_{2\alpha+1}^0$.
    \end{itemize}
\end{lemma}

\begin{proof}
    Given $s$, $k$ and $\alpha$ as in the statement, part (i) implies part (ii) since for any $s$, 
    \[
    s\in D^{\omega\alpha+k}(T)\iff s=\emptyset \lor  \exists t\in \omega^k \ (s^\smallfrown t\in D^{\omega \alpha}(T)).
    \]
    For $\alpha=1$, $s\in D^\omega(T)$ can be expressed as a $\mathbf{\Pi}_2^0$ condition on $T$ as follows: 
    \[
    s\in D^{\omega}(T)\iff s\in \bigcap_k D^k(T)\iff s=\emptyset \lor \forall k \exists t\in \omega^k \ (s^\smallfrown t\in T).
    \]
    For the successor step, suppose (i) holds for some $\alpha$. Then also (ii) holds for $\alpha$ and observe that 
    \[
    D^{\omega(\alpha+1)}(T)= D^{\omega}(D^{\omega\alpha}(T))= \bigcap_k D^k(D^{\omega\alpha}(T))=\bigcap_k D^{\omega\alpha+k}(T).
    \] 
    So by the inductive hypothesis, $s\in D^{\omega(\alpha+1)}(T)$ is a countable intersection of $\mathbf{\Sigma}_{2\alpha+1}^0$ conditions on $T$, and hence is $\mathbf{\Pi}^0_{2(\alpha+1)}$, as desired. 
        
    Now let $\lambda<\omega_1$ be limit and suppose the statement holds for all $\alpha<\lambda$. Then 
    \[
    s\in D^{\omega\lambda}(T)\iff s\in \bigcap _{\beta<\omega\lambda}D^{\beta}(T)\iff s\in \bigcap_k\bigcap_{\alpha<\lambda}D^{\omega\alpha+k}(T),
    \] 
    and hence $s\in D^{\omega\lambda}(T)$ is a $\mathbf{\Pi}_{2\lambda}^0$ condition on $T$.
\end{proof}

\begin{lemma}\label{lemma-upper-bounds} 
    Let $\alpha<\omega_1$ and $k\in \omega$. 
    \begin{itemize}
        \item[(i)] If $\alpha$ is a limit ordinal, then $\lht_{<\omega\alpha}$ is $\mathbf{\Sigma}^0_{2\alpha}$.
        \item[(ii)] If $\alpha$ is a successor ordinal, then $\lht_{<\omega\alpha}$ is $\mathbf{\Sigma}^0_{2\alpha+1}$. 
        \item[(iii)]$\lht_{<\omega\alpha+k+1}$ is $\mathbf{\Pi}^0_{2\alpha+2}$.
    \end{itemize}
\end{lemma}

\begin{proof} 
    First observe that
    \begin{equation*}\label{succ-lt}
        |T|<\omega\alpha+k+1\iff \forall s(s\notin D^{\omega\alpha+k} (T)\lor s\in D^{\omega\alpha+k+1}(T) ).
    \end{equation*}
    So by Lemma~\ref{D-upper-bd}, $\lht_{<\omega\alpha+k+1}$ is $\mathbf{\Pi}_{2\alpha+2}^0$. This proves part (iii).
    
    To prove (i) and (ii), observe that for $\alpha>0$,
    \begin{equation}\label{limit-lt}
        |T|<\omega\alpha \iff \exists \mu<\omega \alpha\ \forall s( s\notin D^\mu (T)\lor s\in D^{\mu+1}(T))
    \end{equation}
    and since $\mu<\omega\alpha $ iff $\mu=\omega\alpha'+m$ for some $\alpha'<\alpha$ and $m\in \omega$, we can rewrite $\eqref{limit-lt}$ as 
    \[
    |T|<\omega\alpha \iff \exists \alpha'<\alpha\ \exists m\in \omega \ \forall s(s\notin D^{\omega\alpha'+m} (T)\lor s\in D^{\omega \alpha'+m+1}(T) ).
    \] 
    So by Lemma~\ref{D-upper-bd}, $\lht_{<\omega\alpha}$ is a countable union of $\mathbf{\Pi}_{2\alpha'+2}^0$ sets for $\alpha'<\alpha$. So $\lht_{<\omega\alpha}$ is $\mathbf{\Sigma}_{2\alpha}^0$ if $\alpha$ is a limit ordinal, and $\mathbf{\Sigma}_{2\alpha+1}^0$ if $\alpha$ is a successor.    
\end{proof}

\subsection{Lower Bounds on Complexity}\label{subsection-lower-bd}

In this subsection, we construct Wadge reductions from $2^\omega$ to $\tr$ (Lemma~\ref{lemma-lower-bound}). These reductions will be used in the proof of Theorem~\ref{thm-complete} to show that the upper bounds from Lemma~\ref{lemma-upper-bounds} are the best possible. We start by introducing some auxiliary operations on trees.

\begin{defn}
    Let $S, T\in \tr$ and $T_n\in \tr$ for each $n\in \omega$. 
    \begin{enumerate}
        \item[(i)]The \textbf{tree sum} of $(T_n)$, denoted $\bigoplus_nT_n$, is the tree consisting of the empty sequence and of all sequences $(n)^\smallfrown t$ with $n\in \omega$, $t\in T_n$.
        \item[(ii)]The \textbf{tree product} of $(T_n)$, denoted $\bigotimes_n T_n$, is the tree consisting of the empty sequence and all sequences $(s_0,\dots, s_m)$, $m\in \omega$ such that for each $k\leq m$, 
        \begin{equation}\label{proj}
            s_k=((t_0^0),(t_1^0,t_0^1),(t_2^0,t_1^1,t^2_0), \dots, (t_k^0,\dots , t_0^k))
        \end{equation}
        and for each $i\leq k$, $(t^i_0,\dots, t^i_{k-i})\in T_i$. 
        \item[(iii)] The \textbf{join} of $S$ and $T$, denoted $S\vee T$, is the tree obtained from the disjoint union $S\sqcup T$ by gluing the roots of $S$ and $T$. Formally, it is a tree on $\omega$ defined by 
        \[
        S\vee T=\{\emptyset\}\cup\{(2n_i)_i:(n_i)_i\in S \}\cup \{(2m_i+1)_i:(m_i)_i\in T\}.
        \]
    \end{enumerate}
\end{defn}

The sum and join of trees on $\omega$ are defined as trees on $\omega$ and we will view the product of trees on $\omega$ as a tree on $\omega$ by fixing a bijection $\omega^{<\omega}\to \omega$. It follows from the definition of the topology on $\tr$ that the maps
    \begin{align*}
        (T_0, T_1, \dots )&\mapsto \bigoplus_n T_n,\\
        (T_0, T_1, \dots )&\mapsto \bigotimes_n T_n,\\
        (S,T)&\mapsto S \vee T,
    \end{align*}
are continuous.

We now establish the relationship between Lusin height and the operations of sums and products of trees.

\begin{lemma}\label{lem-tech-tree-ops}
    Let $T_n\in \tr$ for each $n\in \omega$. 
    \begin{itemize}
        \item[(i)] The tree $\bigoplus_n T_n$ is well-founded if and only if $T_n$ is well-founded for all $n\in \omega$. Moreover,
        \begin{equation}\label{plus}
            \big|\bigoplus _nT_n\big|\leq \sup_n \big( |T_n|+1\big).
        \end{equation}
        \item[(ii)] The tree $\bigotimes_n T_n$ is well-founded if and only if there exists $n_0\in \omega$ such that $T_{n_0}$ is well-founded. Moreover, if $T_{n_0}$ is well-founded, then
        \begin{equation}\label{well-times}
            \big|\bigotimes_nT_n\big|\leq |T_{n_0}|+n_0.
        \end{equation}
        If $T_n$ is ill-founded for all $n\in \omega$, then 
        \begin{equation}\label{ill-times}
            \big|\bigotimes _nT_n\big|\leq \sup_n|T_n|.
        \end{equation}
    \end{itemize}
\end{lemma}

\begin{proof}
    \begin{itemize}
        \item[(i)] The first statement follows immediately from the definition of tree sums. Let $T$ denote $\bigoplus_n T_n$ and note that given $n\in \omega$ and $t\in T_n$ we have
        \[
        \rho_T((n)^\smallfrown t)) = \rho_{T_n}(t).
        \]
        Observe also that if $\rho_T((n))\neq \infty$, then $\rho_T((n)) = |T_n|$. Therefore, \eqref{plus} follows from Proposition~\ref{prop-ht}.
        \item[(ii)] Let $T$ denote $\bigotimes_n T_n$. We start by defining an auxiliary function. For a finite sequence $s$, by $\lh(s)$ we denote its length. For $s\in T$ and $i\in \omega$, let $\proj_i(s)$ be the longest node of $T_i$ which appears in $s$. Formally, $\proj_i(s)=\emptyset$ if $s=\emptyset$ or $i\geq\lh(s)$, otherwise if $s=(s_0,\dots , s_k)$ and $s_k$ is as in \eqref{proj}, then 
        \[
        \proj_i(s)=(t^i_0,t^i_1,\dots, t^i_{k-i})\in T_i.
        \] 
        Note that if $(s_0, s_1, \dots)$ is an infinite branch in $\bigotimes_n T_n$, then for each $n\in \omega$ and $k\geq n$, $\proj_n(s_k)\in T_n $ and $\proj_n(s_k)\subsetneq \proj_n(s_{k+1})$, so $T_n$ is ill-founded. Conversely, if each $T_n$ has an infinite branch $(t^n_i)_{i\in \omega}$, then the sequence $(s_k)_{k\in \omega}$ defined as in \eqref{proj} will be an infinite branch through the tree $\bigotimes_n T_n$. 
        
        For the second part of the proof, we will need the following claim: for $s=(s_0,\dots,s_k)\in T$ and $i\in \{0,\dots ,k\}$,
        \begin{equation}\label{height}
            \rho_T(s)\leq \rho_{T_{i}}(\proj_i(s)).
        \end{equation}
        We prove \eqref{height} by induction on $\rho_{T_i}(\proj_i(s))$. If $\rho_{T_i}(\proj_i(s))=0$, then $\proj_i(s)$ is terminal in $T_i$, so $s$ is terminal in $T$, and hence $\rho_T(s)=0$. Let $\alpha<\omega_1$ and assume that \eqref{height} holds for all $s$ with $\rho_{T_i}(\proj_i(s))<\alpha$. Let $s$ be such that $\rho_{T_i}(\proj_i(s))=\alpha$. If $s$ has no successors in $T$, we are done. If $t\in T$ is an immediate successor of $s$, then $\proj_i(t)$ is an immediate successor of $\proj_i(s)$, and hence $\rho_{T_i}(\proj_i(t))< \alpha$. By the inductive hypothesis,
        \[
        \rho_T(t) \leq \rho_{T_i}(\proj_i(t)).   
        \]
        Therefore, by the definition of the rank of a node, we have 
        \begin{equation*}
            \rho_T(s) \leq \rho_{T_i}(\proj_i(s)),
        \end{equation*}
        which completes the proof of \eqref{height}.
                       
        For any tree $S$ and $n\in \omega$,
        \[
        |S| \leq \sup\{\rho_S(s)+1: s\in S\land  \lh(s)\geq n+1\}+n.
        \]
        In particular, if there exists $n_0\in \omega$ such that $T_{n_0}$ is well-founded, then by \eqref{height},
        \begin{align*}
            |T|&\leq \sup\{\rho_T(s)+1: t\in T\land  \lh(s)\geq n_0+1\}+n_0 \\
            &\leq \sup\{\rho_{T_{n_0}}(\proj_{n_0}(s))+1: s\in T\land \lh(s)\geq n_0+1\}+n_0 \\
            &\leq \sup\{\rho_{T_{n_0}}(t)+1: t\in T\land \lh(t)\geq 1\}+n_0 \\
            &= |T_{n_0}|+n_0.
        \end{align*}
        This proves \eqref{well-times}. On the other hand, if $T$ is ill-founded and $s\in T$ satisfies $\rho_T(s)\neq \infty$, then there is $k_0<\lh(s)$ such that $\rho_{T_{k_0}}(\proj_{k_0}(s))\neq \infty$. So we have
        \[
        \rho_T(s)+1 \leq \rho_{T_{k_0}}(\proj_{k_0}(s))+1 \leq |T_{k_0}|.
        \]
        So by Proposition~\ref{prop-ht}, $|T|\leq \sup_n |T_n|$. \qedhere
    \end{itemize}
\end{proof}

\begin{lemma}\label{main-lem}
    Let $0<\alpha<\omega_1$ and let $C\subseteq 2^\omega$ be $\mathbf{\Pi}^0_{2\alpha}$. There is a continuous function $S_{C}:2^\omega\to \tr$ such that 
    \begin{align}
        x\in C&\implies S_C(x)\in \ilf_{\leq \omega\alpha},\label{red1}\\
        x\notin C&\implies S_C(x)\in \wf_{<\omega\alpha}\label{red2}.
    \end{align}
\end{lemma}

\begin{proof} 
    The proof is by induction on $\alpha$. For ease of notation, let $\Phi(\alpha,C,S)$ abbreviate the statement that $S:2^\omega\to \tr$ is continuous which satisfies \eqref{red1} and \eqref{red2}. Given a clopen set $D\subseteq2^\omega$, define $P_D:2^\omega\to \tr$ such that $P_D(x)$ is a tree with a single infinite branch if $x\in D$ and $P_D(x)=\{\emptyset\}$ if $x\notin D$. Clearly, $P_D$ is continuous for any clopen $D$. Given $C\in \mathbf{\Pi}_2^0$, $C=\bigcap_i\bigcup_jD_{i,j}$ for some clopen $D_{i,j}$'s. We define \[S_{C}(x)=\bigotimes_i\bigoplus_jP_{D_{i,j}}(x).\] The function $S_C$ is continuous by the definition of the topology on $\tr$. To see that $S_C$ satisfies \eqref{red1} and \eqref{red2}, first note that by Lemma~\ref{lem-tech-tree-ops}(i),
    \begin{equation}\label{zero-eq-1}
        x\in \bigcup_jD_{i,j}\implies  \bigoplus_jP_{D_{i,j}}(x)\in \ilf_{\leq 1}
    \end{equation}
    and
    \begin{equation}\label{zero-eq-2}
        x\notin \bigcup_jD_{i,j}\implies  \bigoplus_jP_{D_{i,j}}(x)\in \wf_{= 1}.
    \end{equation}
    So if $x\in C$, then $S_C(x)$ is a product of trees in $\ilf_{\leq 1}$, and hence $S_C(x)\in \ilf_{\leq 1}$ by \eqref{ill-times}. If $x\notin C$, then $S_C(x)$ is a product of trees in $\wf_{=1}$ and (possibly) in $\ilf_{\leq 1}$, so $S_C(x)\in \wf_{<\omega}$ by \eqref{well-times}. 
    We are done with the base case.
    
    Now suppose the statement holds for $0<\alpha<\omega_1$. Let $C\in \mathbf{\Pi}^0_{2\alpha+2}$ and let $B_{i,j}\in \mathbf{\Pi}_{2\alpha}^0$ be such that $C=\bigcap_i\bigcup_j B_{i,j}$. By the inductive hypothesis there are continuous functions $S_{i,j}$ which satisfy $\Phi(\alpha,B_{i,j},S_{i,j})$ for each $i,j\in \omega$. We claim that if $S_C:2^\omega\to \tr$ is defined by 
    \[
    S_C(x)=\bigotimes_i\bigoplus_jS_{i,j}(x),
    \] 
    then $\Phi(\alpha+1,C,S_C)$. The proof is exactly as in the base case. The function $S_C$ is clearly continuous for the same reason. Observe that by the inductive hypothesis and Lemma~\ref{lem-tech-tree-ops}(i),
    \[
    x\in \bigcup_jB_{i,j} \implies \big( \forall j (|S_{i,j}(x)|\leq \omega\alpha)\land \exists j(S_{i,j}(x)\in \ilf) \big)\implies \bigoplus_jS_{i,j}(x)\in \ilf_{\leq \omega\alpha+1}
    \] 
    and 
    \[
    x\notin \bigcup_{j}B_{i,j} \implies \forall j (S_{i,j(x)}\in \wf_{<\omega\alpha})\implies  \bigoplus_jS_{i,j}(x)\in \wf_{\leq \omega\alpha}.
    \]
    Then if $x\in C$, $S_C(x)$ is a product of trees in $\ilf_{\leq \omega\alpha+1}$, so $S_C(x)\in \ilf_{\leq\omega\alpha+1}$ by \eqref{ill-times} and if $x\notin C$, then $S_C(x)$ is a product of trees in $\wf_{\leq \omega\alpha}$ and (possibly) in $\ilf_{\leq\omega\alpha+1}$ hence $S_C(x)\in \wf_{<\omega(\alpha+1)}$ by \eqref{well-times}. This proves the successor case.

    Let $\lambda$ be a limit ordinal and $C\in \mathbf{\Pi}_\lambda^0$. For $n\in \omega$, let $\alpha_n<\lambda$ and $B_n\in \mathbf{\Pi}^0_{2\alpha_n}$ be such that $C=\bigcap_n B_n$. By the inductive hypothesis, for each $n$, there is $S_n:2^\omega \to \tr$ such that $\Phi(\alpha_n,B_n,S_n)$ holds. We define $S_C:2^\omega\to \tr$ by 
    \[
    S_C(x)=\bigotimes_nS_{B_n}(x).
    \]
    If $x\in C$, then for all $n$ we have $x\in B_n$, and so $S_{B_n}(x)\in \ilf_{\leq \omega \alpha_n}$. Hence by \eqref{ill-times}, $S_C(x)\in \ilf_{\leq \omega\lambda}$. If $x\notin C$, then there exists $n_0$ such that $x\notin B_{n_0}$, i.e. $S_{B_{n_0}}(x)\in \wf_{<\omega\alpha_{n_0}}$. Then by \eqref{well-times} we have $S_{C}\in \wf_{<\omega\lambda}$.
\end{proof}

\begin{lemma}\label{lemma-lower-bound}
    Let $\alpha<\omega_1$ and let $A\subseteq 2^\omega$ be $ \mathbf{\Pi}^0_{2\alpha+2}$. For any $\gamma\geq  \omega\alpha$, there is a continuous function $T_A:2^\omega\to \tr$ such that 
    \begin{align}
        x\in A&\implies |T_A(x)|= \gamma ,\label{ta1}\\
        x\notin A&\implies |T_A(x)|=\gamma+1.\label{ta2}
    \end{align}
\end{lemma}

\begin{proof} 
    Fix $\alpha,\gamma$ and $A$ as in the statement and first suppose $\alpha\neq 0$. For $i,j\in \omega$, let $B_{i,j}\in \mathbf{\Pi}^0_{2\alpha}$ be such that $A=\bigcap_i\bigcup_jB_{i,j}$. For ease of notation, let $S_{i,j}$ denote the function $S_{B_{i,j}}$ from Lemma~\ref{main-lem}. Fix any $P\in \wf_{=\gamma}$. We define $T_A:2^\omega\rightarrow \tr$ by
    \[
    T_A(x)=\bigoplus_i(P\vee \bigoplus_j S_{i,j}(x)).
    \] 
    $T_A$ is continuous by the definition of the topology on $\tr$. We will prove that it satisfies \eqref{ta1} and \eqref{ta2}.
    
    Suppose $x\in A$ and fix some $i\in \omega$. There is $j\in \omega$ such that $x\in B_{i,j}$, and hence $S_{i,j}(x)\in \ilf_{\leq \omega \alpha}$ by \eqref{red1}; and if $j'\in \omega$ is such that $x\notin B_{i,j'}$, then $S_{i,j'}(x)\in \wf_{<\omega\alpha}$ by \eqref{red2}. Then $D^{\omega\alpha}(\bigoplus_jS_{i,j})$ is a non-empty pruned tree, call it $S^i$. We have 
    \[
    D^{\gamma}(P\vee \bigoplus_jS_{i,j})=D^{\gamma}(P)\vee S^i=S^i.
    \] 
    Since $i\in \omega$ was arbitrary, $D^{\gamma}(T_A(x))=\bigoplus_iS^i$ is a pruned tree, and hence $|T_A(x)|\leq \gamma$. By Lemma~\ref{lem-tech-tree-ops}(i) we have $|T_A(x)|\geq |P|$, so $|T_A(x)|=\gamma$. 
          
    On the other hand, if $x\notin A$, then there is some $i\in \omega$ such that for every $j\in \omega$, $x\notin B_{i,j}$, and hence $S_{i,j}(x)\in \wf_{<\omega \alpha}$ by \eqref{red2}. Then 
    \[
    D^{\gamma}(\bigoplus _j S_{i,j}(x))=\{\emptyset\}.
    \] 
    By the fact that $|P|=\gamma$ and $\{\emptyset\}\vee\{\emptyset\}=\{\emptyset\}$, we have 
    $D^{\gamma}(T_A(x))=\bigoplus_i\{\emptyset\}$. 
    Clearly 
    \[
    D^\gamma(T_A(x))\neq \{\emptyset\} = D^{\gamma+1}(T_A(x))=D^{\gamma+2}(T_A(x)),
    \] 
    so $|T_A(x)|= \gamma+1$.
    
    Finally, if $\alpha=0$, then for $i, j\in \omega$ let $D_{i,j}$'s be clopen sets such that $A=\bigcap_i\bigcup_jD_{i,j}$. Let $P_{D_{i,j}}$ be defined as in the proof of Lemma~\ref{main-lem}. Define $T_A$ by $T_A(x)=\bigoplus_i(P\vee \bigoplus_j P_{D_{i,j}}(x))$. Then one can prove that $T_A$ satisfies the desired properties by the same argument as above using \eqref{zero-eq-1} and \eqref{zero-eq-2}.
\end{proof}

\subsection{Proofs of Theorems~\ref{thm-complete} and \ref{thm-derivative-operator-complexity}}

\begin{proof}[Proof of Theorem~\ref{thm-complete}]
    Given $\alpha<\omega_1$ and $k\in \omega$, let $\gamma=\omega\alpha+k$. Lemma~\ref{lemma-lower-bound} implies that $\lht_{<\omega\alpha+k+1}$ and $\lht_{=\omega\alpha+k}$ are $\mathbf{\Pi}_{2\alpha+2}^0$-hard, and that $\lht_{=\omega\alpha+k+1}$ is $\mathbf{\Sigma}_{2\alpha+2}^0$-hard. Then by Lemma~\ref{lemma-upper-bounds} we have that $\lht_{<\omega\alpha+k+1}$ and $\lht_{=\omega\alpha}=\lht_{<\omega\alpha+1}\cap(\neg \lht_{<\omega\alpha})$ are $\mathbf{\Pi}_{2\alpha+2}^0$-complete. Taking $\gamma=\omega\alpha+k+1$, Lemma~\ref{lemma-lower-bound} implies that $\lht_{=\omega\alpha+k+1}$ is $\mathbf{\Pi}_{2\alpha+2}^0$-hard. Thus $\lht_{=\omega\alpha+k+1}$ is neither $\mathbf{\Sigma}_{2\alpha+2}^0$ nor $\mathbf{\Pi}_{2\alpha+2}^0$, it is a difference of two $\mathbf{\Pi}^0_{2\alpha+2}$ sets by Lemma~\ref{lemma-upper-bounds}(iii). \qedhere
\end{proof}

\begin{proof}[Proof of Theorem~\ref{thm-derivative-operator-complexity}]
    The function $D^{\omega\alpha+k}$ is $\mathbf{\Sigma}_{2\alpha+1}^0$-measurable by Lemma~\ref{D-upper-bd}(ii). That it is not $\mathbf{\Sigma}_{2\alpha}^0$-measurable follows from Theorem~\ref{thm-complete}(i) and the fact that $\lht_{<\omega\alpha+k+1}$ is the preimage under $D^{\omega\alpha+k}$ of the $\mathbf{\Pi}^0_2$ set 
    \[
    \{T\in \tr: \forall t\in T \ \exists n\in \omega \ (t^\smallfrown(n)\in T)\}\cup \{ \emptyset\}. \qedhere
    \]
\end{proof}

\begin{rem}\label{rem-same}
    Theorems~\ref{thm-complete} and \ref{thm-derivative-operator-complexity} remain true if we take Zafrany's definition of the Lusin derivative \cite[Definition 6.1]{Zafrany1989BorelIV}. To see this, first observe that in this case the same proofs work for Lemma~\ref{D-upper-bd} and Lemma~\ref{lemma-upper-bounds} that are used to calculate upper estimates of the complexity. Moreover, for each $\alpha<\omega_1$ joining a single infinite branch with a tree is a Wadge reduction from $\lht_{=\alpha}$ defined using our formula \eqref{derivative_definition} to the set $\lht_{=\alpha}$ defined using Zafrany's definition \cite[Definition 6.1]{Zafrany1989BorelIV}, so the latter set is not simpler than the former one.
\end{rem}

\section{Application to idealistic equivalence relations}\label{Section Application to idealistic equivalence relations}

\subsection{Countable Abelian $p$-groups} 

By $\cong$ we denote the isomorphism relation of groups. Fix a prime $p$. We say that an abelian group $H$ is a \textbf{$p$-group} if for every $h\in H$ there is $n\in \omega$ with $p^nh=0$. We say that $H$ is \textbf{divisible} if for every $h\in H$ and $n\in \omega$ there is $h'\in H$ with $nh'=h$. We say that $H$ is \textbf{reduced} if $H$ has no nontrivial divisible subgroups. By \cite[Theorem 3]{kaplansky1969infinite}, for every abelian $p$-group $H$ we have $H\cong \Red(H)\oplus \Div(H)$, where $\Div(H)$ is the maximal divisible subgroup of $H$ and $\Red(H)=H/\Div(H)$ is reduced. The group $\Div(H)$ is isomorphic to the direct sum $\mathbb Z(p^\infty)^{(n)}$ of $n$ copies of the Prüfer $p$-group for some $n\in \omega+1$. We refer to $n$ as the \textbf{rank of the maximal divisible subgroup of $H$}.

Define a transfinite sequence $(p^\alpha H)_{\alpha<\omega_1}$ of subgroups of $H$ by setting 
\begin{align*}
    p^0H&=H,\\
    p^{\alpha+1}H &= p(p^\alpha H),\\
    p^\alpha H &= \bigcap_{\xi<\alpha} p^\xi H \text{ if $\alpha$ is a limit ordinal.}
\end{align*}
If $H$ is countable, then there is $\tau<\omega_1$ with $p^\tau H = p^{\tau+1}H$ (and then also $p^\tau H = \Div(H)$). Such $\tau$ is called the \textbf{Ulm length of} $H$, and we will denote it by $l(H)$. For any abelian $p$-group $H$, we define $H[p]=\{h\in H: ph=0\}$. For an ordinal $\alpha<\omega_1$, we define the \textbf{$\alpha$-th Ulm invariant} $U_\alpha(H)\in \omega\cup \{\infty\}$ by 
\[
U_\alpha(H) = \dim_{\mathbb Z_p}((p^\alpha H)[p]/(p^{\alpha+1}H)[p]).
\]
The sequence $(U_\alpha(H))_{\alpha<\omega_1}$ is called the \textbf{Ulm sequence of }$H$. Note that $U_\alpha(H)=0$ for $\alpha\geq l(H)$ and that $U_\alpha(H)=U_\alpha(\Red(H))$ for $\alpha<\omega_1$.

Ulm's theorem \cite[Theorem 14]{kaplansky1969infinite} states that two countable reduced abelian $p$-groups are isomorphic if and only if their Ulm sequences are equal. Consequently, two countable abelian $p$-groups are isomorphic if and only if their Ulm sequences are equal and their maximal divisible subgroups have the same rank. 

\subsection{The Relation $F_\mathbb B$} 

Let $\mathcal L=\{+\}$ be a relational language. For $x\in X_{\mathcal L}=2^{\omega \times \omega \times \omega}$, let $H_x$ denote the $\mathcal{L}$-structure encoded by $x$, i.e., $H_x=(\omega,+_x)$, where \[k+_xm=n\iff x(k,m,n)=1.\] Let $\mathcal{T}$ be the theory of abelian $p$-groups.

\begin{defn}
    Let $\Mod(\mathcal{T})$ be topologized as a subspace of $X_{\mathcal{L}}$. Define $F_{\mathbb B}$ on $\Mod(\mathcal{T})$ by 
    \[xF_{\mathbb B}y \iff \Red(H_x)\cong \Red(H_y).\]
\end{defn}

There is a continuous map $\theta:\Mod(\mathcal{T}) \times \Mod(\mathcal{T}) \to \Mod(\mathcal{T})$ such that for all $x,y\in \Mod(\mathcal{T})$ 
\[
H_{\theta(x,y)}=H_x\oplus H_y.
\]
Indeed, we first define a group operation on $\omega^2$ by 
\[
(n,m)+_{\theta(x, y)} (k,l) = (n+_x k, m+_y l),
\]
and then we transfer this operation back to $\omega$ by a fixed bijection $\varphi:\omega\to \omega^2$.

We proceed by proving some properties of $F_\mathbb B$, following the arguments from \cite{BeckersNotes}.
\begin{prop}\label{becker-alt}
\begin{enumerate}
    \item[(i)] $\Mod(\mathcal{T})$ is an $\cong$-invariant $G_\delta$ subset of $X_{\mathcal L}$.
    \item[(ii)] The relation $F_{\mathbb B}$ is Borel reducible to $ \cong_\mathcal{L}$ (the isomorphism relation of countable $\mathcal{L}$-structures), and hence $F_\mathbb B$ is an analytic equivalence relation all of whose equivalence classes are Borel.
    \item[(iii)] The relation $F_\mathbb{B}$ is idealistic.
\end{enumerate}
\end{prop}

\begin{proof}
\begin{itemize}
    \item[(i)] Note that for an abelian group axioms of existence of the neutral element and of inverse elements can be replaced by an axiom $\forall a \forall b \exists x (ax=b)$. So the conclusion follows from counting quantifiers.
    \item[(ii)] First note that \[xF_{\mathbb B}y\iff H_x\oplus \mathbb Z(p^\infty)^{(\omega)}\cong H_y\oplus \mathbb Z(p^\infty)^{(\omega)}.\] 
    Pick $y$ such that $H_y\cong \mathbb Z(p^{\infty})^{(\omega)}$. Then $\theta(\cdot,y):\Mod(\mathcal{T})\to \Mod(\mathcal{T})$ is the desired Borel (even continuous) reduction from $F_{\mathbb B}$ to $\cong_\mathcal{L}$, and the rest of part (ii) follows from the fact that $\cong_\mathcal{L}$ is an analytic equivalence relation with Borel equivalence classes.
    \item[(iii)] Note that  for $y$ as above, $\theta(\cdot,y):\Mod(\mathcal{T})\to \Mod(\mathcal{T})$ defined above selects an $\cong_\mathcal{L}$-equivalence class within every $F_\mathbb B$-equivalence class in the sense that 
    \[
    xF_{\mathbb B} z \implies \theta(x,y)\cong_\mathcal{L} \theta(z,y)
    \]
    and
    \[
    x F_{\mathbb B} \theta(x,y).
    \]
    Then, since $\cong_\mathcal{L}$ is idealistic as an orbit equivalence relation, $F_{\mathbb B}\supseteq \cong_{\mathcal L}$, and $\theta(\cdot, y)$ is continuous, we can apply \cite[Proposition 2.4]{calderoni2025structuralresultsidealisticequivalence}. \qedhere
\end{itemize}
\end{proof}

Now we give the original definition of Becker's equivalence relation $E_{\mathbb B}$ and prove that it is classwise Borel isomorphic to $F_{\mathbb B}$. 

\begin{defn}
    Let $\widetilde{\mathcal L}=\mathcal L\cup \{a_0,a_1,\dots \}$, where $\{a_0,a_1,\dots \}$ is a countable set of constant symbols. Enumerate the $p$-group $\mathbb Z(p^\infty)^{(\omega)} = \{a_0, a_1, \dots\}$ and let $\widetilde{\mathcal{T}}=\mathcal{T}\cup Diag_{\widetilde{\mathcal L}}\left(\mathbb Z(p^\infty)^{(\omega)}\right)$, where $Diag_{\widetilde{\mathcal L}}\left(\mathbb Z(p^\infty)^{(\omega)}\right)$ denotes the atomic diagram of $\mathbb Z(p^\infty)^{(\omega)}$ as an $\widetilde{\mathcal L}$-structure. The relation $E_\mathbb B$ is defined on $\Mod(\widetilde{\mathcal T})\subseteq 2^{\omega\times \omega \times \omega}\times \omega^\omega$ by
    \[
    (x, c)E_\mathbb B (y, d)\iff x \cong_\mathcal L y.
    \] 
\end{defn}

\begin{prop}\label{prop-alt-becker}
    Relations $F_\mathbb B$ and $E_\mathbb B$ are classwise Borel isomorphic. In particular, $F_\mathbb B$ and $E_\mathbb B$ are Borel bireducible.
\end{prop}

\begin{proof}
    Define $y\in \Mod(\mathcal{T})$ by $y(n,m,k)=1\Leftrightarrow a_n+a_m=a_k$. For $x\in \Mod(\mathcal{T})$, let $e(x)\in \omega$ be the neutral element of $H_x$. We let $c_n(x) = \varphi^{-1}(e(x), n)$. Then the assignment $a_n \mapsto c_n(x), n\in \omega$ is an isomorphism. Therefore $(\theta(x, y), (c_n))\in \Mod(\widetilde{\mathcal T})$. Observe that for $x, z\in \Mod(\mathcal{T})$
    \begin{equation*}
        x F_\mathbb B z \iff (\theta(x,y), (c_n(x))) E_\mathbb B (\theta(z,y), (c_n(z))) 
    \end{equation*}
    and for $(x, c), (z, d)\in \Mod(\widetilde{\mathcal T})$
    \begin{equation*}
        (x, c) E_\mathbb B (z, d) \iff x \cong_\mathcal{L} z \iff x F_\mathbb B z.
    \end{equation*}
    It follows that the maps
    \begin{align*}
        x&\mapsto (\theta(x,y), (c_n(x))),\\
        (x, c)&\mapsto x
    \end{align*}
    are Borel (even continuous) reductions, and it is straightforward to check that their factorings witness the classwise Borel isomorphism of $F_\mathbb B$ and $E_\mathbb B$. 
\end{proof}
 
\subsection{Main Theorem}

We say that $E$ is \textbf{classwise Borel embeddable into} $F$ if there is a Borel $F$-invariant subset $A\subseteq Y$ such that $E$ is classwise Borel isomorphic to $F\restriction A$. 

\begin{thm}\label{thm-Beckers-relation}
Relations $E_{\mathbb B}$ and $F_{\mathbb B}$ are not classwise Borel embeddable into an orbit equivalence relation.  
\end{thm}

By Proposition~\ref{prop-alt-becker} it suffices to prove Theorem~\ref{thm-Beckers-relation} for $F_{\mathbb B}$, which follows immediately from combining Lemmas~\ref{lemma-Becker-Kechris}, \ref{lemma_E0_does_not_reduce}, and \ref{becker-lem-nonexplicit}.

The general Lemma~\ref{lemma-Becker-Kechris} below  is isolated from arguments of Becker \cite{BeckersNotes} and Calderoni and Motto Ros \cite{calderoni2025structuralresultsidealisticequivalence}.
The relation of eventual equality on $2^\omega$, denoted by $\mathbb E_0$, is defined by \[x\mathbb E_0 y\iff \exists k\forall n\geq k\ ( x(n)=y(n)).\]

\begin{lemma}\label{lemma-Becker-Kechris}
    Let $X$ be a Polish space and $F$ be an equivalence relation on $X$. Assume that 
    \begin{itemize}
        \item[(i)] $\mathbb E_0\not\leq_B F$;
        \item[(ii)] for uncountably many limit ordinals $\beta<\omega_1$, there is a non-empty $\mathbf{\Pi}^0_{\beta+1}$ $F$-invariant set $Y(\beta)$ with no $\mathbf{\Pi}^0_{\beta+1}$ $F$-class in $Y(\beta)$.
    \end{itemize} 
    Then $F$ is not classwise Borel embeddable into any orbit equivalence relation. 
\end{lemma}

The following two lemmas concern the equivalence relation $F_{\mathbb B}$. 

\begin{lemma}\label{lemma_E0_does_not_reduce}
    $\mathbb{E}_0 \not\leq_B F_\mathbb B$.
\end{lemma}

\begin{lemma}\label{becker-lem-nonexplicit}
    For uncountably many limit ordinals $\beta<\omega_1$, there is a non-empty $\mathbf{\Pi}^0_{\beta+1}$ $F_{\mathbb B}$-invariant set $Y(\beta)$ with no $\mathbf{\Pi}^0_{\beta+1}$ $F_{\mathbb B}$-class in $Y(\beta)$. 
\end{lemma}

Lemmas~\ref{lemma-Becker-Kechris} and \ref{lemma_E0_does_not_reduce} are proved in Subsections~\ref{Proofs-of-Lemmas} and \ref{subsection lemma E0}, respectively. Lemma~\ref{becker-lem-nonexplicit}  is proved in Subsection~\ref{subsection-tr-gp}; in fact, we prove a more explicit version of this result as Lemma~\ref{becker-lem}. 

\begin{rem}\label{rem-Ulm-classifiability}
    Denote by $\wo$ the set of all well orders on $\omega$ and let $X\subseteq 2^\omega$ be a Polish space. We say that an assignment $f:X\to 2^{<\omega_1}$ is \textbf{$C$-measurable in the codes} if there exists a map $f^*:X\to \wo^\omega$ such that $f^*$ is $C$-measurable (i.e., measurable with respect to the smallest $\sigma$-algebra that contains all Borel sets and is closed under the Souslin operation) and for $x\in X$ 
    \[
    f(x) = \{\type(f^*(x)(n)): n\in \omega\},
    \]
    where for $w\in \wo$, $\type (w)$ is the countable ordinal isomorphic to $w$. 

    Consider the following condition on an equivalence relation $F$ on a Polish space $X$:
    \begin{itemize}
        \item[(i')] There exists a map $f:X\to 2^{<\omega_1}$ which is $C$-measurable in the codes and such that $xFy \Leftrightarrow f(x)=f(y)$.
    \end{itemize}
    
    An Ulm sequence $(U_\alpha(H))$ of a countable abelian $p$-group $H$ takes values in $\omega\cup \{\infty\}$ and is eventually zero. Given an Ulm sequence $(U_\alpha(H))$ we can define $(U'_\alpha(H))\in 2^{<\omega_1}$ in the following way: for countable $\beta$ we let $U'_{\omega\beta}(H)=1$ if $U_\beta(H) = \infty$ and for $n>0$, $U'_{\omega\beta+n}(H) = 1$ if $U_\beta(H) = n-1$. The map $x\mapsto \Ulm(x)=(U'_\alpha(H_x))$ satisfies condition (i') with $X=\Mod(\mathcal{T})$ and $F=F_{\mathbb B}$, therefore (using \cite[Theorem 3.4.4]{Becker_Kechris_1996}) we can prove a version of Lemma~\ref{lemma-Becker-Kechris} with (i) replaced by (i') and apply it to prove Theorem~\ref{thm-Beckers-relation}. To see that the map $x\mapsto \Ulm(x)$ is $C$-measurable in the codes, let us define $\Ulm^*: \Mod(\mathcal{T})\to \wo^\omega$. For a countable abelian $p$-group $H$ and $h\in H$, we define the \textbf{height of $h$ in $H$} (denoted by $\htt_H(h))$ as the unique $\alpha<\omega_1$ with $h\in p^\alpha H \setminus p^{\alpha+1}H$ if it exists and $\infty$ otherwise (we set $\alpha<\infty$ for any ordinal $\alpha$). This gives us a preordering $\leq_x$ on $\omega$ for each $x\in \Mod(\mathcal{T})$
    \[
    n\leq_x m \iff \htt_{H_x}(n)\leq \htt_{H_x}(m).
    \]
    Fix a well order $\prec$ of $\omega\times \omega$ such that $\type(\prec)=\omega$. For each $\alpha<l(H_x)$, we let $m(\alpha)\in \omega$ be minimal (with respect to the natural order) such that $\htt_{H_x}(m(\alpha))=\alpha$. Then we define 
    \[
    V_x = \{(m(\alpha), k): \alpha<l(H_x) \land U'_{\omega \alpha+k}(H_x)=1\}
    \]
    and let $\{(m_n,k_n)\}_{n \in \omega}$ be a $\prec$-increasing enumeration of $V_x$. Note that the order $\leq_x$ restricted to $\{m(\alpha): \alpha<\htt_{H_x}(m_n)\}$ is a well order of type $\htt_{H_x}(m_n)$. We define $\Ulm^*(x)(n)$ to be a well order of $\omega$ of type $\omega \htt_{H_x}(m_n)+k_n$ obtained from $V_x$ and $\leq_x$ in a continuous way. The main part of proving that $\Ulm^*$ is a $C$-measurable assignment is showing that $x\mapsto \leq_x$ is $C$-measurable, which is done by constructing a continuous reduction $(x, n)\mapsto T(x,n)\in \tr$ such that 
    \[
    \htt_{H_x}(n)\leq \htt_{H_x}(m)\iff |T(x, n)|\leq |T(x, m)|.
    \]
    A similar argument is used in the proof of Lemma~\ref{lemma-comeager}. We skip the details here. 
    
    This remark corresponds to the Glimm--Effros-type dichotomy which states that conditions (i) and (i') are equivalent if $F$ is an orbit equivalence relation \cite[Theorem 3.4.4]{Becker_Kechris_1996}. Note however that $F_{\mathbb B}$ is not an orbit equivalence relation by Theorem~\ref{thm-Beckers-relation}. It should also be noted that in the statement of the dichotomy, condition "$\mathbb E_0 \not\leq_B F$" is replaced by "$\mathbb E_0 \not\sqsubseteq_c F"$ (where the latter means "there is no continuous embedding of $\mathbb E_0$ into $F$"), but those are equivalent for any $F$. Indeed, if $f:2^\omega\to X$ witnesses $\mathbb E_0\leq_B F$, then it also witnesses $\mathbb E_0\leq_B F'$, where $F'=(f\times f)(\mathbb E_0)\cup \{(x, x): x\in X\}$ is Borel, and by applying Harrington--Kechris--Louveau dichotomy \cite[Theorem 1.1]{Harrington} we then have $\mathbb E_0\sqsubseteq_c F'$, and so also $E_0\sqsubseteq_c F$. 
\end{rem}

\subsection{Proof of Lemma~\ref{lemma-Becker-Kechris}}\label{Proofs-of-Lemmas}

Before we proceed to prove Lemma~\ref{lemma-Becker-Kechris}, we prove the following simple fact about $\mathbf{\Sigma}^0_\alpha$-measurable functions. The proof is the same as in \cite[Proof of Lemma 3.13]{calderoni2025structuralresultsidealisticequivalence}, but we include it here for the sake of completeness.

\begin{lemma}\label{simple-lemma}
    Let $X, Y$ be Polish spaces, $\alpha<\omega_1$ and $f:X\to Y$ be a $\mathbf{\Sigma}_\alpha^0$-measurable function. For every $\beta\geq\alpha\omega$, $f^{-1}(A)\in \mathbf{\Sigma}^0_\beta(X)$ whenever $A\in \mathbf{\Sigma}^0_\beta(Y)$. 
\end{lemma}

\begin{proof}
    First, we inductively prove that, for each $\gamma<\omega_1$,  
    \begin{equation}\label{E:sat}
        A\in \mathbf{\Sigma}^0_{1+\gamma}(Y)\implies f^{-1}(A)\in \mathbf{\Sigma}^0_{\alpha+\gamma}(X).
    \end{equation} 
    For $\gamma=0$, \eqref{E:sat} follows from $f$ being $\mathbf{\Sigma}^0_\alpha$-measurable. Fix some $\gamma<\omega_1$ and assume that \eqref{E:sat} holds for all $\delta<\gamma$. Let $A\in \mathbf{\Sigma}^0_{1+\gamma}(Y)$. Then for some $B_n\in \mathbf{\Sigma}^0_{1+\gamma_n}(Y)$, $n\in \omega$, with $\gamma_n<\gamma$, we have
    \[
    A = \bigcup_n \big( Y\setminus B_n\big).
    \]
    Then 
    \[
    f^{-1}(A) = \bigcup_n \big( X\setminus f^{-1}(B_n)\big),
    \] 
    which is, by the inductive hypothesis, in $\mathbf{\Sigma}^0_{\alpha+\gamma}(X)$, and \eqref{E:sat} is proved. 

    By \eqref{E:sat}, given $n\in \omega$, if $A\in \mathbf{\Sigma}^0_{1+\alpha n}(Y)$, then 
    \[
    f^{-1}(A) \in \mathbf{\Sigma}^0_{\alpha+ \alpha n}(X)=\mathbf{\Sigma}^0_{\alpha (n+1)}(X). 
    \]
    In particular, $f^{-1}(A)\in \mathbf{\Sigma}^0_{\alpha \omega}(X)$ for $A\in \mathbf{\Sigma}^0_{\alpha \omega}(Y)$, so the lemma holds for $\beta=\alpha\omega$. We proceed inductively. Assume that the lemma holds for all $\delta<\omega_1$ such that $\alpha\omega \leq \delta <\beta$. Let $A\in \mathbf{\Sigma}^0_{\beta}(Y)$. Then for some $B_n\in \mathbf{\Sigma}^0_{\beta_n}(Y)$, $n\in \omega$, with $\alpha\omega\leq\beta_n<\beta$,  we have
    \[
    A = \bigcup_n \big( Y\setminus B_n\big).
    \]
    Then 
    \[
    f^{-1}(A) = \bigcup_n \big( X\setminus f^{-1}(B_n)\big),
    \] 
    which is in $\mathbf{\Sigma}^0_{\beta}(X)$ by the inductive hypothesis. 
\end{proof}

\begin{proof}[Proof of Lemma~\ref{lemma-Becker-Kechris}]
    Suppose towards a contradiction that there is a relation $E$ induced by a Borel action of a Polish group on a standard Borel space $Z$ and a classwise Borel embedding $f:X\to Z$ of $F$ into $E$ with a classwise Borel inverse $g:[f(X)]_E\to X$, where $[f(X)]_E$ is the $E$-saturation of $f(X)$. Since $[f(X)]_E$ is $E$-invariant and, by the definition of classwise Borel embeddability, Borel, we can assume without loss of generality that $Z=[f(X)]_E$. By \cite[5.2.1 Theorem]{Becker_Kechris_1996} we can further assume without loss of generality that $Z$ is a Polish space and $E$ is induced by a continuous action.
    
    Since both $f$ and $g$ are Borel functions, there is $\alpha<\omega_1$ such that $f, g$ are $\mathbf{\Sigma}^0_\alpha$-measurable. By Lemma~\ref{simple-lemma}, for every $\gamma\geq \alpha\omega$ and $A\in \mathbf{\Sigma}_{\gamma}^0(Z)$, $f^{-1}(A)\in \mathbf{\Sigma}_{\gamma}^0(X)$ and for all $B\in \mathbf{\Sigma}_{\gamma}^0(X)$, $g^{-1}(B)\in \mathbf{\Sigma}_{\gamma}^0(Z)$. 
    
    By assumption (ii), there is a limit $\beta\geq\alpha\omega$ and a non-empty $\mathbf{\Pi}^0_{\beta+1}$ $F$-invariant set $Y(\beta)\subset X$ with no $\mathbf{\Pi}^0_{\beta+1}$ $F$-classes in $Y(\beta)$. Then $g^{-1}(Y(\beta))$ is a non-empty $\mathbf{\Pi}^0_{\beta+1}$ $E$-invariant subset of $Z$. Observe that $\mathbb{E}_0 \not\leq_B E\restriction g^{-1}(Y(\beta))$, since otherwise we would have $\mathbb{E}_0 \leq_B E \leq_B F$, (where $g$ witnesses $E\leq_B F$), which contradicts assumption (i). Therefore by \cite[Corollary 5.1.10]{Becker_Kechris_1996} there is a $\mathbf{\Pi}^0_{\beta+1}$ $E$-class $A \subseteq g^{-1}(Y(\beta))$. Then $f^{-1}(A)$ is a $\mathbf{\Pi}^0_{\beta+1}$ $F$-class in $Y(\beta)$, a contradiction. 
\end{proof}
\subsection{Proof of Lemma~\ref{lemma_E0_does_not_reduce}}\label{subsection lemma E0} 

To prove Lemma~\ref{lemma_E0_does_not_reduce} we need Lemma~\ref{lemma-comeager}, in which the following relationship between $p$-groups and trees is employed. For $x\in \Mod(\mathcal{T})$ which encodes a $p$-group $H_x=(\omega, +_x)$, we define a tree $\Tree(x)$ on $\omega$ in the following way: $(n)\in \Tree(x)$ if and only if order on $n$ in $H_x$ is equal to $p$; further, if $(n_0, \dots, n_k)\in \Tree(x)$ and for some $n$, $pn=n_k$, then $(n_0, \dots, n_k, n)\in \Tree(x)$. The assignment $x\mapsto \Tree(x)$ is continuous and for all $x\in \Mod(\mathcal{T})$
\begin{equation}\label{length=height}
    l(H_x) = |\Tree(x)|.
\end{equation}
Moreover, it is clear that $\Tree(x)$ is well-founded if and only if $H_x$ is reduced. To prove \eqref{length=height}, we show that for all $\alpha<\omega_1$ and all $(n_0, \dots, n_k)\in \Tree(x)$
\begin{equation}\label{eq-height-and-subgroups}
    \rho_{\Tree(x)}(n_0, \dots, n_k)\geq \alpha \Leftrightarrow n_k \in p^\alpha H_x.
\end{equation}
Equivalence \eqref{eq-height-and-subgroups} clearly holds for $\alpha=0$. Let $\alpha<\omega_1$ and assume that \eqref{eq-height-and-subgroups} holds for all $\gamma<\alpha$. Let $t=(n_0, \dots, n_k)\in \Tree(x)$ and assume that $\rho_{\Tree(x)}(t)\geq \alpha$. Let $\gamma<\alpha$. There is $n\in \omega$ with $\rho_{\Tree(x)}(t^\smallfrown(n))\geq \gamma$. Then, by the inductive hypothesis, $n\in p^\gamma H_x$. Therefore, $n_k = pn \in p^{\gamma+1}H_x$. Since $\gamma<\alpha$ was arbitrary, $n_k\in p^\alpha H_x$. The converse implication is proved analogously.

For a function $h:X\to Y$, we define $h\times h:X\times X\to Y\times Y$ by 
\begin{equation}
    (h\times h )(x, x') = (h(x), h(x')).
\end{equation}

\begin{lemma}\label{lemma-comeager}
    Let $f:2^\omega\to \Mod(\mathcal{T})$ be a Borel function. There is $\xi<\omega_1$ such that the set
    \[
    \{x\in 2^\omega: l(H_{f(x)})<\xi\}
    \]
    is comeager. 
\end{lemma}

\begin{proof}
    Let $U_n, n\in \omega$ be a clopen basis of $2^\omega$. If we show that the set 
    \[
    {\leq_l}= \{(x, y)\in 2^\omega\times 2^\omega: l(H_{f(x)})\leq l(H_{f(y)})\}
    \]
    has the Baire property, then for every $n$ by \cite[Exercise 8.49]{kechris1995classical} there is $\xi_n<\omega_1$ with 
    \[
    \{x\in U_n: l(H_{f(x)})<\xi_n\}
    \]
    non-meager and we can take $\xi=\sup_n \xi_n$. To show that $\leq_l$ has the Baire property, observe that
    \[
    \leq_l = (h\times h)^{-1}(\{(S, T)\in \tr \times \tr: |S|\leq |T|\}),
    \]
    where $h(x)=\Tree(f(x))$, and for $S, T\in \tr$, $|S|\leq |T| $ is equivalent to
    \[
    \forall s \big((s\in S \land s\neq \emptyset \land S_s\in \wf )\implies \exists t (t\in T \land t\neq \emptyset \land T_t\in \wf \land |S_s|\leq |T_t|)\big),
    \]
    which is a Baire measurable condition on $(S, T)$ since $T\mapsto |T|$ is a $\mathbf{\Pi}^1_1$ rank on $\wf$.    
\end{proof}

\begin{proof}[Proof of Lemma~\ref{lemma_E0_does_not_reduce}]
    Suppose towards a contradiction that $f:2^\omega\to \Mod(\mathcal{T})$ is a Borel reduction witnessing $\mathbb E_0\leq_B F_\mathbb B$. Then by Lemma~\ref{lemma-comeager} there is $\xi<\omega_1$ with 
    \[
    \{x\in 2^\omega: l(H_{f(x)})<\xi\}
    \]
    comeager. For $\alpha<\xi$ and $n\in \omega\cup \{\infty\}$, the sets
    \[
    \{x\in 2^\omega: U_\alpha(H_{f(x)})=n\}
    \]
    are preimages under Borel function $f$ of sets defined by $\mathcal{L}_{\omega_1, \omega}$-formulas \cite[Lemma 2.2]{BARWISE197025} and therefore, by \cite[Proposition 16.7]{kechris1995classical}, these sets are Borel. Since they are also $\mathbb E_0$-invariant, by the second topological $0$-$1$ law \cite[Theorem 8.47]{kechris1995classical}, for each $\alpha<\xi$ there is $n_\alpha$ such that 
    \[
    \{x\in 2^\omega: U_\alpha(H_{f(x)})=n_\alpha\}
    \]
    is comeager. Therefore the set 
    \[
    \{x\in 2^\omega: l(H_{f(x)})<\xi \land \forall \alpha<\xi \ (U_\alpha(H_{f(x)})=n_\alpha)\}
    \]
    is comeager as a countable intersection of comeager sets. However, this is a single $\mathbb E_0$-class, so it is a countable set, a contradiction. 
\end{proof}

\subsection{Proof of Lemma~\ref{becker-lem-nonexplicit}}\label{subsection-tr-gp}

In this subsection we prove an explicit version of Lemma~\ref{becker-lem-nonexplicit}, namely Lemma~\ref{becker-lem}. We first recall a standard way of associating a $p$-group $G(T)$ to a tree $T$: $G(T)$ is the free abelian group generated by the elements of $T$ modulo the words $\emptyset$ and $ps-t$ for all $s,t\in T$, where $t$ is an immediate predecessor of $s$. Here $ps$ denotes the sum $s+\dots +s$ with $p$-many $s$'s. This assignment can be formalized by a continuous function $g:\tr\to \Mod(\mathcal{T})$ satisfying $H_{g(T)}\cong G(T)$ for any $T\in \tr$. With this in mind, for $s\in T$ we denote by $ps$ the immediate predecessor of $s$ in $T$ for $s\neq \emptyset$, and we let $p\emptyset=\emptyset$. 

\begin{lemma}\label{lemma-tree-ulm-length}
    For $T\in \tr$ and $\alpha<\omega_1$ we have $l(G(T))=|T|$.
\end{lemma}

\begin{proof}   
    This follows from the fact that $G(D^\alpha(T))\cong p^\alpha (G(T))$ for all ordinals $\alpha$, which can be proved by induction on $\alpha$ by using \cite[(3.4) and (3.5)]{crawley-hales}.
\end{proof}

\begin{lemma}\label{lemma-tilde-hard}
    For every ordinal $\alpha<\omega_1$ and $k\in \omega$, the sets
    \[
    A^{\omega\alpha+k}=\{x\in \Mod(\mathcal{T}):l(H_x)=\omega\alpha+k\},
    \]
    \[
    K^{\omega\alpha+k}=\{x\in A^{\omega\alpha+k}: \forall \xi<\omega\alpha+k \ ( U_\xi(H_x)=\infty)\}
    \]
    are $\mathbf{\Pi}^0_{2\alpha+2}$-hard. 
\end{lemma}

\begin{proof}
    By Lemma~\ref{lemma-tree-ulm-length} we have that $g^{-1}(A^{\omega\alpha+k})=\lht_{=\omega\alpha+k}$. It follows from parts (ii) and (iii) of Theorem~\ref{thm-complete} that 
    $A^{\omega\alpha+k}$ is $\mathbf{\Pi}^0_{2\alpha+2}$-hard. 
    
    An easy calculation shows that for two countable abelian $p$-groups $H_1$ and $H_2$, we have 
    \begin{equation}\label{lenght-of-sum}
        l(H_1\oplus H_2)=\max\{l(H_1),l(H_2)\}
    \end{equation} and 
    \begin{equation}\label{inv-of-sum}
    U_\xi(H_1\oplus H_2)=U_\xi(H_1)+U_\xi(H_2)
    \end{equation}
    for $\xi<\omega_1$. Moreover, it follows from \cite[Theorem 4.1]{BARWISE197025} that $K^\gamma \neq \emptyset$ for any $\gamma<\omega_1$. Take $y\in K^{\omega\alpha+k}$. Recall that $\theta:\Mod(\mathcal{T})^2\to \Mod(\mathcal{T})$ is a continuous function  such that $H_{\theta(x,y)}\cong H_x\oplus H_{y}$. By Lemma~\ref{lemma-tree-ulm-length}, \eqref{lenght-of-sum}, and \eqref{inv-of-sum} we have
    \begin{align*}
        |T|\leq \omega\alpha+k\iff l(g(T))\leq \omega\alpha+k \iff \theta(g(T),y)\in K^{\omega\alpha+k}.
    \end{align*}
     Therefore, by Theorem~\ref{thm-complete}(i), and the continuity of $g$ and $\theta$, $K^{\omega\alpha+k}$ is $\mathbf{\Pi}^0_{2\alpha+2}$-hard. 
\end{proof}

\begin{lemma}\label{becker-lem}
    Let $\beta<\omega_1$ be such that $\beta=\omega\beta$ and let 
    \[
    Y(\beta)=\{x\in \Mod(\mathcal{T}):l(H_x)\geq \beta \land \forall \alpha<\beta \ (U_{\alpha}(H_x)=\infty)\}.
    \]
    Then
    \begin{itemize}
        \item[(i)] $Y(\beta)$ is a non-empty $\mathbf{\Pi}^0_{\beta+1}$ $F_{\mathbb B}$-invariant set,
        \item[(ii)] every $F_{\mathbb B}$-class contained in $Y(\beta)$ is $\mathbf{\Pi}^0_{\beta+2}$-hard.
    \end{itemize}
\end{lemma}

\begin{proof}
        The set $Y(\beta)$ is non-empty by \cite[Theorem 4.1]{BARWISE197025}, and it clearly is $F_{\mathbb B}$-invariant. To finish proof of part (i) we will show that $Y(\beta)$ is actually $\mathbf{\Pi}^0_\beta$. Indeed, observe that for fixed $\alpha<\omega_1$ and $k, n\in \omega$, $n\in p^{\omega\alpha+k}H_x$ is a $\mathbf{\Sigma}^0_{2\alpha+1}$ condition on $x$ by an easy calculation as in the proof of Lemma~\ref{D-upper-bd}. Therefore, $l(H_x)\geq \beta$, which is equivalent to
        \[
        \forall \xi<\beta \ \exists n\in \omega \ (n\in p^{\xi}H_x\setminus p^{\xi+1}H_x),
        \]
        is a $\mathbf{\Pi}^0_\beta$ condition for a limit ordinal $\beta<\omega_1$. 
        Furthermore, note that for fixed $\alpha<\omega_1$ and $N\in \omega$ the condition $U_\alpha(H_x)\geq N$ is equivalent to
        \begin{align*}
            \exists n_1, \dots, n_N \big(& n_i\in p^\alpha H_x \land pn_i=e(x) \text{ for each } i\\ 
            &\text{and no non-trivial linear combination of $n_1, \dots, n_N$ is in }p^{\alpha+1}H_x \big),
        \end{align*}
        where $e(x)$ is the neutral element of $H_x$. So by the previous observations the condition
        \[
        \forall \alpha <\beta \ (U_\alpha(H_x)=\infty)
        \]
        is also a $\mathbf{\Pi}^0_\beta$ condition for limit $\beta<\omega_1$. 
        
        For part (ii), fix some $F_{\mathbb B}$-class $C$ contained in $Y(\beta)$. Every element of $C$ codes a group of the same Ulm sequence $(U_\alpha)_{\alpha<\omega_1}$ with Ulm length $\lambda$ for some $\lambda<\omega_1$. 
        
        If $U_\alpha=\infty$ for all $\alpha<\lambda$, then $C= K^\lambda$. Note that for $\lambda=\omega\alpha_\lambda+k_\lambda$, where $k_\lambda\in \omega$, since $\lambda \geq \beta$ and $\beta=\omega\beta$, we have $\alpha_\lambda\geq \beta$. Therefore $C$ is $\mathbf{\Pi}_{\beta+2}^0$-hard by Lemma~\ref{lemma-tilde-hard}.
        
        Otherwise, let $\gamma<\lambda$ be the least ordinal such that $U_\gamma< \infty$. Fix $A\in \mathbf{\Pi}_{\beta+2}^0(2^\omega)$. Let $T_A$ be the function from Lemma~\ref{lemma-lower-bound} with $\alpha=\beta$ and note that $\gamma\geq \beta=\omega\beta$ by definition of $Y(\beta)$. So we have 
        \begin{align*}
            x\in A&\implies |T_A(x)|= \gamma,\\
            x\notin A&\implies |T_A(x)|=\gamma+1.
        \end{align*}
        Now set $f(x)=g(T_A(x))$ and pick any $y\in C$. It suffices to show that
        \begin{equation}\label{claim}
            x\in A\iff\theta(f(x),y)\in C.
        \end{equation}
        
        To see \eqref{claim}, suppose first that $x\in A$. Then by Lemma~\ref{lemma-tree-ulm-length} we have $l(H_{f(x)})=|T_A(x)|$. So $l(H_{f(x)})=\gamma<l(H_y)$, and hence $l(H_{\theta(f(x),y)})=l(H_y)$ by \eqref{lenght-of-sum}. Using \eqref{inv-of-sum} we get that $U_\alpha(H_{f(x)}\oplus H_{y})=U_\alpha(H_y)$ since $U_\alpha(H_y)=\infty$ if $\alpha<\gamma$ and $U_\alpha(H_{f(x)})=0$ if $\alpha\geq \gamma$. It follows from Ulm's theorem that $\Red(H_y)\cong  \Red(H_{\theta(f(x),y)})$ and therefore $\theta(f(x),y)\in C$. 
        
        Now suppose $x\notin A$. Then $l(H_{f(x)})=|T_A(x)|=\gamma+1$. We claim that for any countable abelian $p$-group $H$, if $l(H)=\gamma+1$ then $U_\gamma(H)\neq 0$. Indeed, this is easy to see from the definition of $U_\gamma$ if $H$ is reduced. Since $H\cong \Red(H)\oplus \Div(H)$, the conclusion follows from \eqref{inv-of-sum}. So we have that $U_\gamma(H_{f(x)})\neq 0$. Then again by \eqref{inv-of-sum}, $U_\gamma(H_{\theta(f(x),y)})=U_\gamma(H_{f(x)})+U_\gamma(H_{y})>U_\gamma(H_y)$ since $U_\gamma(H_y)<\infty$. Therefore, $\Red(H_{\theta(f(x),y)})\not\cong \Red(H_y)$ and so $\theta(f(x),y)\notin C$, as desired.
\end{proof}

\bibliographystyle{amsplain}

\bibliography{impbib}

\end{document}